\documentclass{article}
\usepackage{graphicx} 

\usepackage[utf8]{inputenc}
\usepackage{amsmath}
\usepackage{amsthm}
\usepackage{amssymb}
\usepackage{mathtools}
\usepackage{bm}
\usepackage{enumitem}
\usepackage{graphicx}
\usepackage{tikz}
\usetikzlibrary{shapes.geometric, positioning, backgrounds}
\tikzset{source/.style={circle,fill=gray!70,draw,minimum size=0.5cm,inner 
sep=0pt}}
\tikzset{non-source/.style={circle,draw,minimum size=0.5cm,inner 
sep=0pt}}

\usepackage[a4paper,top=3cm,bottom=3cm,left=3cm,right=3cm,marginparwidth=4cm]{geometry}
\usepackage{todonotes}
\usepackage[numbers]{natbib}
\usepackage[ruled, linesnumbered]{algorithm2e}
\usepackage{hyperref}
\usepackage{cleveref}
\usepackage[indent,skip=\medskipamount]{parskip}

\newtheorem{theorem}{Theorem}
\newtheorem{lemma}{Lemma}
\newtheorem{corollary}{Corollary}

\newtheorem{conjecture}{Conjecture}

\theoremstyle{definition}

\DeclarePairedDelimiter\ceil{\lceil}{\rceil}

\newcommand{\Bound}{\ceil{\sqrt{n}}}

\title{The Graph Burning Conjecture is true for trees without degree-2 vertices}
\author{Yukihiro 
Murakami\thanks{\href{mailto:y.murakami@tudelft.nl}{y.murakami@tudelft.nl}}}
\date{%
	Department of Applied Mathematics, Delft University of Technology, the 
	Netherlands\\[2ex]%
	\today}

\begin{document}

\maketitle

\begin{abstract}
Graph burning is a discrete time process which can be used
to model the spread of social contagion. One is initially given a graph of 
unburned vertices. At each round (time step), one vertex is burned; unburned 
vertices with at least one burned neighbour from the previous round also 
becomes burned. The burning number of a graph 
is the fewest number of rounds required to burn the graph. It has been 
conjectured that for a graph on~$n$ vertices, the burning number is at 
most~$\Bound$. We show that the graph burning conjecture is true for trees 
without degree-2 vertices. 
\end{abstract}

\section{Introduction}

Given is a finite simple connected graph. Initially, all vertices are unburned. 
One new vertex, called a \emph{source}, may be burned every round. If a vertex 
was burned in the previous round, all its unburned neighbours become burned in 
the current round. Once a vertex is burned, it cannot be unburned. The 
\emph{burning number}~$b(G)$ of a graph~$G$ is the minimum number of rounds 
required to burn every vertex. See \Cref{fig:BurningNumber} for examples. 

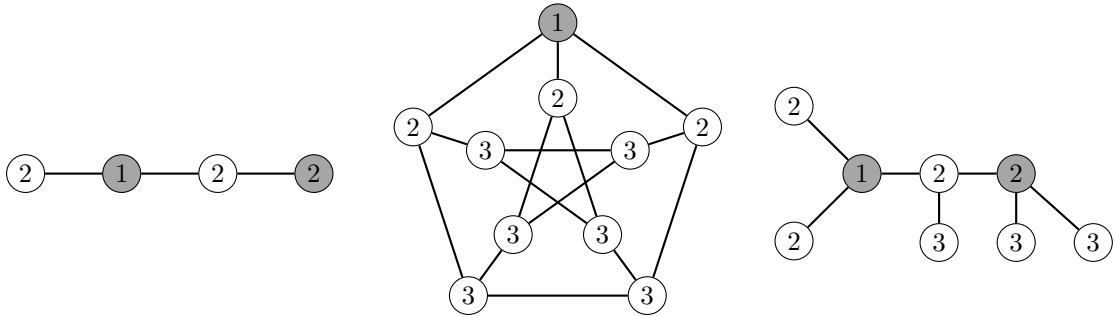
\begin{figure}[h!]
	\centering
	\begin{tikzpicture}
		\node[non-source] (1) {$2$};
		\node[source] (2) [right = 0.75cm of 1] {$1$};
		\node[non-source] (3) [right = 0.75cm of 2] {$2$};
		\node[source] (4) [right = 0.75cm of 3] {$2$};
		
		\path[draw,thick]
		(1) edge node {} (2)
		(2) edge node {} (3)
		(3) edge node {} (4);
		\begin{scope}[xshift=7cm]
			
			\node[source] (1) at (90:2) {$1$};
			\node[non-source] (2) at (162:2) {$2$};
			\node[non-source] (3) at (234:2) {$3$};
			\node[non-source] (4) at (306:2) {$3$};
			\node[non-source] (5) at (18:2) {$2$};
			\node[non-source] (6) at (90:1) {$2$};
			\node[non-source] (7) at (162:1) {$3$};
			\node[non-source] (8) at (234:1) {$3$};
			\node[non-source] (9) at (306:1) {$3$};
			\node[non-source] (10) at (18:1) {$3$};
			
			\path[draw,thick]
			(1) edge node {} (2)
			(1) edge node {} (5)
			(2) edge node {} (3)
			(3) edge node {} (4)
			(4) edge node {} (5)
			(1) edge node {} (6)
			(2) edge node {} (7)
			(3) edge node {} (8)
			(4) edge node {} (9)
			(5) edge node {} (10)
			(6) edge node {} (8)
			(6) edge node {} (9)
			(7) edge node {} (9)
			(7) edge node {} (10)
			(8) edge node {} (10);
		\end{scope}
		\begin{scope}[xshift=11cm]
			
			\node[source] (2) {$1$};
			\node[non-source] (1) [above left = 0.75cm of 2] {$2$};
			\node[non-source] (3) [below left = 0.75cm of 2] {$2$};
			\node[non-source] (4) [right = 0.5cm of 2] {$2$};
			\node[non-source] (5) [below = 0.4cm of 4] {$3$};
			\node[source] (6) [right = 0.5cm of 4] {$2$}; 
			\node[non-source] (7) [below = 0.4cm of 6] {$3$};
			\node[non-source] (8) [right = 0.5cm of 7] {$3$};
			
			\path[draw,thick]
			(1) edge node {} (2)
			(2) edge node {} (3)
			(2) edge node {} (4)
			(4) edge node {} (5)
			(4) edge node {} (6)
			(6) edge node {} (7)
			(6) edge node {} (8);
		\end{scope}
	\end{tikzpicture}
	\caption{The path on 4 vertices, the Petersen graph, and a HIT. The path 
	has burning number 2 and the other graphs have burning number 3. Sources 
	are highlighted in gray, and numbers on the vertices indicate the round in 
	which they become burned. Note that only one source suffices to 
	achieve the burning number for the Petersen graph, and two sources suffice 
	to achieve the burning number for the HIT.}
	\label{fig:BurningNumber}
\end{figure}

Such a problem was studied initially by Alon~\cite{alon1992transmitting} to 
solve a communication problem, where they showed that the burning number of an 
$n$-dimensional hypercube is $\ceil{n/2}+1$. Independently, the problem was 
investigated under the name \emph{graph burning} by Bonato et al. to model the
spread of social contagion and influential nodes~\cite{bonato2014burning}. The 
graph burning problem has since garnered much attention, with results focussed 
on bounds~\cite{Bastide2021}, complexity~\cite{bessy2017burning}, and 
approximation schemes~\cite{Martinsson2023}. The problem has also been 
investigated for the directed graph variant~\cite{janssen2020burning}. For a 
general summary on the state of the art, we refer the interested reader to the 
following survey~\cite{bonato2020survey}.

One particularly interesting family for graph burning is the path graph. It was 
shown by Bonato et al. that a path on~$n$ vertices has burning 
number exactly~$\Bound$~\cite{Bonato2016}. As paths are seemingly the most 
sparse with respect to graph burning, Bonato et al. gave the following 
conjecture in the same paper.
\begin{conjecture}[\emph{Burning Number Conjecture}  
\cite{Bonato2016}]\label{conj:BNC}
	Let~$G$ be a connected graph on~$n$ vertices. Then~$b(G)\le \Bound$.
\end{conjecture}
In the same paper, they gave a reduction which shows that if \Cref{conj:BNC} is 
true for trees, then it is true for all graphs.
\begin{lemma}[Corollary 2.5 of \cite{Bonato2016}]\label{lem:Subtree}
	For a graph~$G$, we have that
	\[b(G) = \min\{b(T): \text{$T$ is a spanning tree of~$G$}\}.\]
\end{lemma}

While \Cref{conj:BNC} has not been proven in full, it has been shown to be true 
for certain graph classes, such as spider graphs~\cite{Bonato2019} and 
some~$p$-caterpillar graphs~\cite{Hiller2020}. The conjecture is also known to 
hold for large enough graphs with minimum degree 3 or 4~\cite{Bastide2021}, 
and it is known to hold asymptotically~\cite{Norin2022}. Recently, it was also 
shown to be true for trees where every internal vertex is of degree 
$3$~\cite{das2023burning}.

In this paper, we show that \Cref{conj:BNC} is true for 
\emph{homeomorphically irreducible trees (HITs)}\footnote{HITs also go by other 
names, such as \emph{series-reduced trees}, \emph{irreducible trees}, and 
\emph{topological trees}.}, 
which are trees without degree-2 vertices (\Cref{thm:TreesWOD2Vertices}). 
HITs are counterparts to paths, cycles, and hamiltonicity, as the former 
minimizes and the latter maximizes the number of degree-2 vertices.

\section{Burning number of HITs}\label{sec:HITbn}

Let~$G$ be a finite simple connected graph, and let~$xy$ be a bridge in~$G$. 
Let us denote the component that contains~$x$ upon deleting~$xy$ as~$G_x(xy)$. 
For vertices~$x,y$ in~$G$, let~$d_G(x,y)$ denote the length of a shortest path 
between~$x$ and~$y$.

\begin{lemma}\label{lem:TheRightBridge}
	Let~$n\ge6$. Any tree~$T$ on~$n$ vertices contains a vertex~$x$ with 
	neighbours~$v_1,\ldots, v_k$ such that~$|T_x(xv_k)|\ge 2\Bound-1$ 
	and~$|T_{v_i}(xv_i)|< 2\Bound-1$ for~$i\in[k-1] = \{1,2,\ldots, k-1\}$. 
\end{lemma}
\begin{proof}
	Let~$T$ be a tree on~$n$ vertices. Take a leaf~$\ell$ and its 
	neighbour~$x$. We have that
	\[|T_x(x\ell)| = n-1 \ge 2\Bound-1,\]
	where the last inequality follows as~$n\ge6$. Let~$v_1,\ldots, v_k$ denote 
	the neighbours of~$x$ where~$v_k=\ell$. If~$|T_{v_i}(xv_i)|< 
	2\Bound-1$ for~$i\in[k-1]$, then we are done. Otherwise, there 
	exists a neighbour, without loss of generality,~$v_1$, 
	where~$|T_{v_1}(xv_1)| \ge 2\Bound-1$. Let~$u_1,\ldots, u_j$ 
	denote the neighbours of~$v_1$, where~$u_j=x$. If~$|T_{u_i}(v_1u_i)|< 
	2\Bound-1$ for~$i\in[j-1]$, then we are done. Otherwise, we 
	continue in the same manner; such a process must terminate as~$T$ is finite.
\end{proof}

\begin{lemma}\label{lem:BoundOnInt}
	Let~$n\in\mathbb{N}^{>0}$. Let~$T$ be a HIT where $|T|\le2n-1$. Then 
	$T$ contains at most~$n-2$ internal vertices.
\end{lemma}
\begin{proof}
	In a HIT, there are at least two more leaves than internal vertices. 
	Suppose for a contradiction that $T$ contains at least~$n-1$ internal 
	vertices. Then~$T$ contains at least~$n+1$ leaves, and we have our required 
	contradiction.
\end{proof}

\subsection{Graph Burning}

We shall formalize the notion of graph burning. Let~$G$ be a graph, and 
let~$S=(x_1,x_2,\ldots,x_k)$ be a sequence of vertices, called \emph{sources}. 
Initially, in 
round~$0$, all vertices start as an unburned vertex. In round~$i$, burn the 
vertex~$x_i$ (if it is not burned already) as well as any unburned vertices 
that have burned neighbours in round~$i-1$, for~$i\in[k]$. If at the end of 
round~$k$, all vertices of~$G$ are burned, we call~$S$ a \emph{burning 
sequence} for~$G$. The \emph{burning number}~$b(G)$ of~$G$ is the length of a 
shortest burning sequence for~$G$. Note that burning sequences do not in 
general have unique lengths. 

In the proof of the following results, we require a notion of graph burning 
where multiple sources can be burned in round 1. For~$U\subseteq V(G)$ and 
vertices~$x_i\in V(G)$, let~$M = (U\cup\{x_1\}, x_2, \ldots, x_k)$ be a 
sequence. In round 1, burn all vertices in the set~$U\cup\{x_1\}$; in round 
$i$, proceed as done in the traditional burning sequence. We call~$M$ a 
\emph{modified burning sequence} for~$G$ if all vertices of~$G$ are burned 
after round~$k$. The \emph{modified burning number}~$b^{U}(G)$ of~$G$ is the 
length of a shortest modified burning sequence for~$G$, with some 
set~$U\subseteq V(G)$.

\begin{lemma}\label{lem:HITwithD2Burned}
	Let~$T$ be a tree with one degree-$2$ vertex~$v$, and let~$T'$ be the HIT 
	obtained from~$T$ by smoothing~$v$. Then
	\[b^{\{v\}}(T) \le b(T').\]
\end{lemma}
\begin{proof}
	Let~$(x_1,\ldots, x_k)$ be a (not necessarily optimal) burning sequence  
	for~$T'$. We claim that $(\{v,x_1\}, x_2,\ldots, x_k)$ is a modified 
	burning sequence for~$T$. This would imply that any burning sequence 
	for~$T'$ yields a modified burning sequence of the same length for~$T$, 
	from which we may conclude that $b^{\{v\}}(T) \le b(T')$.
	
	Suppose for a contradiction that $(\{v,x_1\}, x_2,\ldots, x_k)$ is not a 
	modified burning sequence for~$T$. This means there exists a vertex~$w$ 
	in~$T$ that is not burned at the end of round~$k$. Clearly~$w$ cannot be 
	one of the sources. Since~$(x_1,\ldots, x_k)$ is a burning sequence 
	for~$T'$,~$w$ is a burned vertex at the end of round~$k$. Suppose that~$w$ 
	becomes burned in~$T'$ in round~$j$ for some~$j\le k$. Since~$w$ is not a 
	source, there must be a source~$x_i$ with~$i<j$ such that~$d_{T'}(w,x_i) 
	=j-i$.
	
	Because~$w$ remains unburned in~$T$, we must have that $d_{T}(w,x_i) 
	>j-i$. Then the path from~$w$ to~$x_i$ in~$T$ must contain the 
	vertex~$v$, as this is the only difference between trees~$T$ and~$T'$. It 
	follows that~$d_{T}(w,x_i) = j-i+1$. But then~$d_{T}(w,v)\le j-i$. This 
	would mean that~$w$ becomes a burned vertex in~$T$ no later than 
	round~$1+j-i$, since~$v$ is burned in round 1. Since~$1+j-i\le k$, this 
	means that~$w$ is burned in~$T$ at the end of round~$k$. This gives the 
	required contradiction.
\end{proof}

\begin{theorem}\label{thm:TreesWOD2Vertices}
	Let~$T$ be a HIT on~$n$ vertices. Then~$b(T) \leq \Bound$.
\end{theorem}
\begin{proof}
	We prove by induction on the number of vertices~$n$. For the base case, we 
	consider the 4 HITs possible for~$n\le 5$. If~$n=1$, the HIT is a tree on a 
	single vertex, which has burning number 1. For~$n=2$, the HIT is a tree 
	with a single edge between two leaves, which has burning number 2. 
	For~$n=4$ and~$n=5$, we have star graphs, which have burning number 2 by 
	choosing the internal vertex to be the first source. We may now assume 
	that~$T$ is a HIT on $n\ge 6$ vertices, and that for all HITs with number 
	of vertices fewer than~$n-1$, the theorem holds.
	
	By \Cref{lem:TheRightBridge}, $T$ contains a vertex~$x$ with 
	neighbours~$v_1,\ldots, v_k=y$ such that~$|T_x(xy)|\ge 2\Bound-1$ 
	and~$|T_{v_i}(xv_i)|< 2\Bound-1$ for~$i\in[k-1]$. Let~$i\in[k-1]$. 
	We claim that the longest distance from $x$ to any vertex 
	in~$H_i:=V(T_{v_i}(xv_i))$ is at most~$\Bound -1$. 
	
	Consider the induced subtree~$T_i:=T[H_i\cup \{x\}]$. 	
	Since~$|H_i| <2\Bound-1$, it follows that~$|T_i| \le 2\Bound -1$. 
	By~\Cref{lem:BoundOnInt}, there can be at most $\Bound-2$ internal vertices 
	in~$T_i$. Since~$x$ is a leaf in~$T_i$, at most $\Bound-2$ vertices 
	of~$H_i$ are internal vertices in~$T$. It follows immediately that any 
	longest path in~$T$ from~$x$ to a vertex in~$H_i$ is of distance at 
	most~$\Bound$. As this is true for all~$i\in[k-1]$, it follows that the 
	distance from~$x$ to every vertex in~$T_x(xy)$ is at most $\Bound$.
	
	We now claim that~$T$ can be burned in at most~$\Bound$ rounds, by 
	burning~$x$ in round 1. Observe that upon burning~$x$ in round 1, with no 
	additional sources in~$T_x(xy)$, all vertices of~$T_x(xy)$ will be burned 
	by the end of round $\Bound$. Indeed, this occurs since the distance 
	from~$x$ to every vertex in $T_x(xy)$ is at most $\Bound$. 
	
	The vertex~$y$ will become burned in round 2, as it is a neighbour of~$x$. 
	It remains to show that~$b^{\{y\}}(T_y(xy))\le \Bound-1$. Suppose first 
	that~$y$ is a degree-$2$ vertex in~$T_y(xy)$. Let~$T'$ denote the HIT 
	obtained by smoothing~$y$ in~$T_y(xy)$. By \Cref{lem:HITwithD2Burned}, we 
	have that~$b^{\{y\}}(T_y(xy))\le b(T')$. On the other hand, suppose now 
	that~$y$ is not a degree-$2$ vertex. Then~$T_y(xy)$ itself is a HIT and we 
	have the inequality~$b^{\{y\}}(T_y(xy))\le b(T_y(xy))$. Now we also have 
	that
	\[|T'| < |T_y(xy)| = n - |T_x(xy)| \le n-2\Bound + 1 \le (\Bound-1)^2.\]
	It follows by induction hypothesis that~$b(T') \le \Bound -1$ 
	and~$b(T_y(xy))\le \Bound -1$. Therefore, we obtain $b^{\{y\}}(T_y(xy))\le 
	\Bound -1$, and we are done.
\end{proof}

It follows immediately from \Cref{lem:Subtree} and \Cref{thm:TreesWOD2Vertices} 
that the burning number conjecture is true for all graphs that contain a 
\emph{homeomorphically irreducible spanning tree (HIST)} (\Cref{cor:HIST}). 

\begin{corollary}\label{cor:HIST}
	Let~$G$ be a graph on~$n$ vertices with a HIST. Then~$b(G)\le \Bound$.
\end{corollary}

Apart from appearing in the movie `Good Will Hunting', HITs have been studied 
extensively within graph theory as a fundamental structure, with initial 
results concerning enumeration of labelled and unlabelled 
HITs~\cite{harary1959number}. 
HITs are counterparts to paths, cycles, and hamiltonicity, which all maximize 
the number of degree-2 vertices. Naturally, the problem of determining whether 
a graph contains a HIST has been thoroughly investigated. While 
this decision problem has been shown to be 
NP-complete~\cite{albertson1990graphs}, finding conditions for when a graph 
contains a HIST is an ongoing research 
area~\cite{zhai2019homeomorphically,furuya2023new}.
For example, it is known that every connected and locally connected graph with 
at least 4 vertices contains a HIST~\cite{chen2012homeomorphically}.

\section{Concluding Remarks}

We have shown that HITs (trees without degree-$2$ vertices) with~$n$ vertices 
satisfy the burning 
number 
conjecture, i.e., that they have burning number at most~$\Bound$. Consequently, 
any graph that contains a HIST also satisfies the 
burning number conjecture. 
HITs are antithetical to paths, as the former contains no degree-$2$ vertices 
and the latter maximizes them. It is interesting to see that in the extreme 
cases, with respect to the number of degree-2 vertices, the burning number 
conjecture holds; to prove the conjecture in full, it remains to show true for
the intermediate instances.

The proof of \Cref{thm:TreesWOD2Vertices} hinges on \Cref{lem:BoundOnInt}, 
which is not true for trees with degree-$2$ vertices. We suspect therefore that 
another proof strategy is necessary for proving the full conjecture. 
Nevertheless, we show that \Cref{thm:TreesWOD2Vertices} can be used to find 
bounds on the burning number for general graphs by adding leaves.
\begin{corollary}\label{cor:AddLeaves}
	Let~$T$ be a tree on~$n$ vertices. If~$T$ has~$d$ degree-2 vertices, then
	\[b(T)\le \ceil{\sqrt{n+d}}.\]
\end{corollary}
\begin{proof}
	Add a leaf to every degree-$2$ vertex and apply 	
	\Cref{thm:TreesWOD2Vertices}.
\end{proof}
Unfortunately, this bound underperforms the recent bound presented 
in~\cite{das2023burning} for large enough~$n$, 
which shows that for a tree on~$n$ vertices and~$d$ degree-2 vertices,
\[b(T) \le \ceil{\sqrt{n+d+8}}-1.\]

In future, it may be of interest to find bounds on the burning number of 
general graphs by looking at spanning trees with the maximum number of leaves 
({\sc Maximum Leaf Spanning Tree} or equivalently, {\sc Connected Dominating 
Set}; the problems are known to be NP-complete~\cite{lemke1988maximum}, but 
some bounds are known~\cite{ding2001spanning,Alon2023}). This 
would give a bound on the number of degree-$2$ vertices, which can 
be used in combination with the above burning number bound and 
\Cref{lem:Subtree}.

\paragraph{Acknowledgement}
The author would like to thank Mark Jones for introducing him to the graph 
burning problem, and Paul Bastide for giving helpful comments on the manuscript.


\newcommand{\etalchar}[1]{$^{#1}$}

\end{document}